\documentclass[12pt]{amsart}
\usepackage{amsmath,amssymb,amsbsy,amsfonts,amsthm,latexsym,amsopn,amstext,
                            amsxtra,euscript,amscd}

\newfont{\teneufm}{eufm10}
\newfont{\seveneufm}{eufm7}
\newfont{\fiveeufm}{eufm5}
\newfam\eufmfam
                 \textfont\eufmfam=\teneufm \scriptfont\eufmfam=\seveneufm
                 \scriptscriptfont\eufmfam=\fiveeufm
\def\bbbc{{\mathchoice {\setbox0=\hbox{$\displaystyle\rm C$}\hbox{\hbox
to0pt{\kern0.4\wd0\vrule height0.9\ht0\hss}\box0}}
{\setbox0=\hbox{$\textstyle\rm C$}\hbox{\hbox
to0pt{\kern0.4\wd0\vrule height0.9\ht0\hss}\box0}}
{\setbox0=\hbox{$\scriptstyle\rm C$}\hbox{\hbox
to0pt{\kern0.4\wd0\vrule height0.9\ht0\hss}\box0}}
{\setbox0=\hbox{$\scriptscriptstyle\rm C$}\hbox{\hbox
to0pt{\kern0.4\wd0\vrule height0.9\ht0\hss}\box0}}}}
\def\bbbq{{\mathchoice {\setbox0=\hbox{$\displaystyle\rm
Q$}\hbox{\raise
0.15\ht0\hbox to0pt{\kern0.4\wd0\vrule height0.8\ht0\hss}\box0}}
{\setbox0=\hbox{$\textstyle\rm Q$}\hbox{\raise
0.15\ht0\hbox to0pt{\kern0.4\wd0\vrule height0.8\ht0\hss}\box0}}
{\setbox0=\hbox{$\scriptstyle\rm Q$}\hbox{\raise
0.15\ht0\hbox to0pt{\kern0.4\wd0\vrule height0.7\ht0\hss}\box0}}
{\setbox0=\hbox{$\scriptscriptstyle\rm Q$}\hbox{\raise
0.15\ht0\hbox to0pt{\kern0.4\wd0\vrule height0.7\ht0\hss}\box0}}}}
\def\bbbt{{\mathchoice {\setbox0=\hbox{$\displaystyle\rm
T$}\hbox{\hbox to0pt{\kern0.3\wd0\vrule height0.9\ht0\hss}\box0}}
{\setbox0=\hbox{$\textstyle\rm T$}\hbox{\hbox
to0pt{\kern0.3\wd0\vrule height0.9\ht0\hss}\box0}}
{\setbox0=\hbox{$\scriptstyle\rm T$}\hbox{\hbox
to0pt{\kern0.3\wd0\vrule height0.9\ht0\hss}\box0}}
{\setbox0=\hbox{$\scriptscriptstyle\rm T$}\hbox{\hbox
to0pt{\kern0.3\wd0\vrule height0.9\ht0\hss}\box0}}}}
\def\bbbs{{\mathchoice
{\setbox0=\hbox{$\displaystyle     \rm S$}\hbox{\raise0.5\ht0\hbox
to0pt{\kern0.35\wd0\vrule height0.45\ht0\hss}\hbox
to0pt{\kern0.55\wd0\vrule height0.5\ht0\hss}\box0}}
{\setbox0=\hbox{$\textstyle        \rm S$}\hbox{\raise0.5\ht0\hbox
to0pt{\kern0.35\wd0\vrule height0.45\ht0\hss}\hbox
to0pt{\kern0.55\wd0\vrule height0.5\ht0\hss}\box0}}
{\setbox0=\hbox{$\scriptstyle      \rm S$}\hbox{\raise0.5\ht0\hbox
to0pt{\kern0.35\wd0\vrule height0.45\ht0\hss}\raise0.05\ht0\hbox
to0pt{\kern0.5\wd0\vrule height0.45\ht0\hss}\box0}}
{\setbox0=\hbox{$\scriptscriptstyle\rm S$}\hbox{\raise0.5\ht0\hbox
to0pt{\kern0.4\wd0\vrule height0.45\ht0\hss}\raise0.05\ht0\hbox
to0pt{\kern0.55\wd0\vrule height0.45\ht0\hss}\box0}}}}
\def\bbbz{{\mathchoice {\hbox{$\sf\textstyle Z\kern-0.4em Z$}}
{\hbox{$\sf\textstyle Z\kern-0.4em Z$}}
{\hbox{$\sf\scriptstyle Z\kern-0.3em Z$}}
{\hbox{$\sf\scriptscriptstyle Z\kern-0.2em Z$}}}}

 \newtheorem{thm}{Theorem}
 \newtheorem{cor}[thm]{Corollary}
 \newtheorem{lem}[thm]{Lemma}
  \newtheorem{prob}[thm]{Problem}
 
 \theoremstyle{definition}
 
 \theoremstyle{remark}

\def\cA{{\mathcal A}}

\def\cE{{\mathcal E}}

\def\cI{{\mathcal I}}

\def\cN{{\mathcal N}}

\def\cP{{\mathcal P}}

\def\({\left(}
\def\){\right)}
\def\[{\left[}
\def\]{\right]}
\def\<{\langle}
\def\>{\rangle}

\def\Z{\mathbb{Z}}

\def\mand{\qquad\mbox{and}\qquad}

\begin{document}

\title[Quadratic Non-residues in Short Intervals]
{Quadratic Non-residues in Short Intervals}

\author[S. V. Konyagin] {Sergei V.~Konyagin}

\address{Steklov Mathematical Institute,
8, Gubkin Street, Moscow, 119991, Russia}
\email{konyagin@mi.ras.ru}

\author[I. E. Shparlinski] {Igor E. Shparlinski}

\address{Department of Pure Mathematics, University of New South Wales,
Sydney, NSW 2052, Australia}
\email{igor.shparlinski@unsw.edu.au}

\begin{abstract}
We use the Burgess bound and combinatorial sieve 
to obtain an upper bound on the number of primes $p$ in a
dyadic interval $[Q,2Q]$ for which a given interval $[u+1,u+\psi(Q)]$
does not contain a quadratic non-residue modulo $p$. The bound is
nontrivial for any function $\psi(Q)\to\infty$ as $Q\to\infty$.
This is an analogue of the well known estimates on the smallest quadratic
non-residue modulo $p$ on average over primes $p$, which corresponds to
the choice $u=0$.
\end{abstract}



\maketitle

\section{Introduction}

\subsection{Motivation and background}

For a prime $p\ge 3$ we denote by $n(p)$ the smallest quadratic non-residue
modulo $p$. The best known upper bound
$n(p) \le p^{1/4e^{1/2} + o(1)}$ is due to Burgess~\cite{Burg},
while it  is expected that $n(p) = p^{o(1)}$, which is
widely known as a {\it Conjecture of Vinogradov\/}.

Bound of this type, and in fact much more precise,
are also known. For example,
conditionally on the Generalised
Riemann Conjecture, we have $n(p) = O(\log^2 p)$ for any prime $p$, see~\cite[Theorem~13.11]{MoVa}.

Furthermore, unconditionally, using the large sieve method,
 Erd{\H o}s~\cite{Erd} has established that
$$
\frac{1}{\pi(x)} \sum_{p\le x} n(p) \to \sum_{k=1}^{\infty} \frac{p_k}{2^k},
\qquad x \to \infty,
$$
where, as usual $\pi(x)$ denotes the number of primes $p \le x$
and $p_k$ denotes the $k$th prime.
This instantly implies that the inequality
$n(p)\le \psi(p)$ holds  for almost all primes $p$ (that is,
for all but $o(x/\log x)$ primes $p\le x$, as $x \to \infty$),
where $\psi$ is an arbitrary function with
$\psi(z) \to \infty$ as $z \to \infty$.

On the other hand, by a result of Graham and   Ringrose~\cite{GrRi},
there is an absolute constant $C > 0$
such that for infinitely many primes $p$  all
nonnegative integers $z \le C \log p \log \log \log p$
are quadratic residues modulo $p$.

Another {\it Conjecture of Vinogradov\/} is the bound
$d(p) = p^{o(1)}$, where $d(p)$ is the longest sequence
of consecutive quadratic residues modulo $p$.
It seems that this conjecture received less attention than the
one  about the smallest quadratic non-residue.
In particular, the only known result about $d(p)$ is
the bound
$d(p)\le   p^{1/4 + o(1)}$,
which is due to Burgess~\cite{Burg} as well.
It is still unknown
whether the Generalised
Riemann Conjecture or the large sieve method (or any other
standard methods and conjectures) can lead to a better estimate
on $d(p)$ for at least almost all primes.
This naturally leads to the following:

\begin{prob} Assuming  the Generalised
Riemann Conjecture, show that for some constant $\gamma < 1/4$ the bound
$d(p) < p^\gamma$ holds for almost all primes $p$.
\end{prob}

In fact, it is still unknown whether  $d(p) = o(p^{1/4})$ for an
infinite sequence of primes.

Our main goal here is to attract more attention to the function $d(p)$
and also make a modest step towards better understanding
the distribution of quadratic non-residues.

We also denote by $n_k(p)$ the $k$th quadratic non-residue modulo $p$,
and consider the gaps
$\Delta_k(p) = n_{k+1}(p) - n_k(p)$,  $k=1, \ldots, (p-3)/2$.

It is shown in~\cite[Lemma~2]{DES} that for any fixed $\varepsilon > 0$
and $h\ge p^\varepsilon$
$$
\# \{k =1, \ldots, (p-3)/2~:~ \Delta_k(p) \ge h\} \le p^{1/2 + o(1)} h^{-2}.
$$
which, via partial summation, 
     leads 
to the estimate
$$
S(h,p)  = \sum_{\substack{j=1\\ \Delta_k(p) \ge h}}^{(p-3)/2} \Delta_k(p)
\le p^{1/2 + o(1)} h^{-1}.
$$

We also note that a result of Garaev, Konyagin and Malykhin~\cite[Theorem~2]{GKM},
in particular,
gives an asymptotic formula for the average values of
the $\gamma$-powers of gaps between quadratic residues modulo $p$
for $0 < \gamma < 4$. This can easily be extended to the same
estimate for the gaps  between quadratic non-residues modulo $p$.

     \subsection{Main result}

Let $d_u(p)$ be smallest $h$ such that there exist a quadratic non-residue in the
interval $\cI = [u+1, u+h]$.
Clearly
$$
n(p) = d_u(p) \mand d(p) = \max_{u \in \Z} d_u(p).
$$
So estimating $d_u(p)$ for a given $u$ can be considered as an intermediate
question between estimating $n(p)$ and $d(p)$.

Here we estimate
$d_u(p)$, uniformly over $u$,  for almost all primes $p$.
It is more convenient to work with primes from dyadic
intervals $[Q, 2Q]$.

\begin{thm}
\label{thm:dup} Let $\psi$ be an arbitrary function with
$\psi(z) \to \infty$ as $z \to \infty$.
For any sufficiently large real positive $Q$,
for any  integer $u\le 2Q$, for the set $\cE_u(\psi, Q)$ of primes $p \in [Q, 2Q]$
with
$$
d_u(p) > \psi(p)
$$
we have  $\cE_u(\psi, Q) = o(Q/\log Q)$ uniformly in $u$.
\end{thm}

\section{Preliminaries}

\subsection{General notation}

Throughout the paper, the implied constants in the symbols ``$O$'',
``$\ll$'' and ``$\gg$'' may occasionally, where obvious,
depend on the real positive parameters $\varepsilon$ and $\eta$
and are absolute otherwise.  We recall that
the expressions $A=O(B)$, $A \ll B$ and  $B\gg A$
are each equivalent to the
statement that $|A|\le cB$ for some constant $c$.

We always use the letter $p$, with or without subscripts, to denote a prime number,
while $k$, $m$, $n$ and $q$ always denote positive integer numbers.

As usual, we use $\varphi(k)$ is the Euler function.

\subsection{Burgess bound}

We now recall the Burgess bound for some of multiplicative
characters modulo arbitrary integers, see~\cite[Theorems~12.5 and~12.6]{IwKow}.
In fact we only need it for sums of Jacobi symbols.

\begin{lem}
\label{lem:Burg} For any integers $q \ge M \ge 1$, where
$q\ge 2$ is not a perfect square, we have
$$
\left|\sum_{m\le M} \(\frac{m}{q}\) \right| \le M^{1-1/\nu} q^{(\nu+1)/4\nu^2 + o(1)},
$$
with $\nu =1,2,3$.
\end{lem}

In particular, Lemma~\ref{lem:Burg}  implies:

\begin{cor}
\label{lem:BurgSimpl}
 For any $\varepsilon > 0$ there exists some $\delta> 0$ such that
for any integers
 $M \ge  q^{1/3+\varepsilon}$,
where $q\ge 2$ is not a perfect square, we have
$$
\left|\sum_{m\le M} \(\frac{m}{q}\) \right| \le M^{1-\delta}
$$
\end{cor}

\subsection{Integers with a prescribed
multiplicative structure}

Now given some $\eta> 0$ we denote by $\cP(\eta,M)$ the set of positive
integers $m\le M$ which do not have prime divisors $p \le M^\eta$.
It is well known that  for any fixed $\eta > 0$ we have
\begin{equation}
\label{eq:Card P}
 |\cP(\eta, M)| \le c_0  \frac{M}{\eta \log M}
\end{equation}
for some absolute constants $c_0> 0$,
see, for example,~\cite[Section~III.6.2, Theorem~3]{Ten}.

We now recall the so-called {\it fundamental
lemma of the combinatorial sieve\/}, see, for example,~\cite[Section~I.4.2, Theorem~3]{Ten}.

For a finite set of integers $\cA$ and a set of primes $\cP$
we denote
$$
P(y) = \prod_{\substack{p \in \cP\\p\le y}} p
$$
and
$$
S(\cA, \cP, y) = \# \{a\in \cA~:~ \gcd(a,P(y))=1\}.
$$

\begin{lem}\label{lem:FL}
Assume that for  a finite set of integers $\cA$ and a set of primes $\cP$
there exist
a non-negative multiplicative function $\omega(d)$, a real $X$ and
positive constants $\alpha$ and $A$ such that:
\begin{itemize}
\item for any $d\mid P(y)$, we have
$$\#  \{a\in \cA~:~ a \equiv 0 \pmod d\} = X\frac{\omega(d)}{d} + R_d;$$
\item for any real $v > w \ge 2$ we have
$$
\prod_{w\le p \le v} \(1 -\frac{\omega(p)}{p} \) < \(\frac {\log v}{\log w}\)^\alpha \(1 + \frac{A}{\log w}\).
$$
\end{itemize}
Then uniformly  for $\cA$, $X$, $y$ and $u\ge 1$
$$
S(\cA, \cP, y) = X \prod_{p\mid P(y)}\(1 -\frac{\omega(p)}{p} \) \(1 + O(u^{-u/2})\)
+ O\(\sum_{\substack{d \mid P(y)\\ d \le y^u}} |R_d|\).
$$
\end{lem}

We also need the following well-known statement which follows from
the standard inclusion-exclusion argument and the classical bound on the
number of integer divisors of $q$.

\begin{lem}
\label{lem:Erat} For any integers $q \ge M \ge 1$,  we have
$$
\# \left\{1 \le m \le M~:~ \gcd(m,q) = 1\right\} = \frac{\varphi(q)}{q} M + O(q^{o(1)}).
$$
\end{lem}

The following asymptotic formula for the number of square-free integers
in a short interval is a very special case of a much more general
result of Tolev~\cite[Theorem~1.3]{Tol} (which we apply with $r=2$, $l_1=1$, $l_2=2$),
which in turn extends and
generalises a result of Filaseta and
Trifonov~\cite{FiTr}.

\begin{lem}
\label{lem:SF} For any fixed $\varepsilon> 0$
and real $h \ge u^{1/5+\varepsilon}$,
the interval $[u+1,u+h]$ contains $\(A+ o(1)\)h$
square-free integers $n$ for which $n+1$ is also square-free,
where
$$
A = \prod_{p~\text{prime}}\(1 - \frac{2}{p^2}\).
$$
\end{lem}

\begin{cor}
\label{cor:SF} For any fixed $\varepsilon> 0$
and real $u\ge h \ge u^{1/5+\varepsilon}$,
the interval $[u+1,u+h]$ contains at least $\(A+ o(1)\)h$
odd square-free integers $n$.
\end{cor}

Note, that Corollary~\ref{cor:SF} is much stronger than what
we actually need. Namely, any result with $\alpha < 1/2$ instead
of $1/5$ and arbitrary $A> 0$ is sufficient for our purposes.

\subsection{Character sums with integers from  $\cP(\eta,M)$}

We now consider the sets
$$
\cP_\pm (\eta,M,q)= \left\{m \in \cP(\eta,M)~:~ \(\frac{m}{q}\) =\pm 1\right\}.
$$

\begin{lem}
\label{lem:Ppm}
 For any $\varepsilon > 0$ there exists some  $\eta_0> 0$ such that
     for any positive $\eta < \eta_0$ and integers
$M \ge  q^{1/3+\varepsilon}$,
where $q\ge 2$ is not a perfect square, we have
$$
\left|\cP_\pm (\eta,M,q)  -
\frac{1}{2} M \prod_{p\le M^\eta} \(1- \frac{1}{p}\) \right| \le
C  \eta^{\eta^{-1/2}/4-1} \frac{M}{\log M} +
O\(M^{1-\eta}\),
$$
where $C$  is an absolute constant.
\end{lem}

\begin{proof} We see from Corollary~\ref{lem:BurgSimpl} and Lemma~\ref{lem:Erat}
that for any positive integer $d< q^{\varepsilon/2}$ with $\gcd(d,q)=1$ we have
\begin{equation}
\begin{split}
\label{eq:Ad}
\# \left\{1 \le m \le M~:~ d\mid m  \text{ and } \(\frac{m}{q}\) =\pm 1\right\}&\\
 = \frac{\varphi(q)}{2dq}M & + R(q,M,d) ,
 \end{split}
\end{equation}
where
\begin{equation}
\label{eq:RqMd}
R(q,M,d) =  O((M/d)^{-\delta})
\end{equation}
for some $\delta> 0$ depending only on $\varepsilon$.

We now set $\eta_0 = \delta^2/4$
and apply Lemma~\ref{lem:FL}
with  $u=\eta^{-1/2}$, $y = M^\eta$ and
$$
\omega(d) = \left\{  \begin{array}{ll}
1,& \text{if $\gcd(d,q)=1$};\\
0,& \text{if $\gcd(d,q)>1$}.
\end{array} \right.
$$
We also assume that $\eta$ is small enough so that
$$
y^u =  M^{\eta^{1/2}} \le q^{\varepsilon/2}
$$
so~\eqref{eq:Ad} applies to all positive integers $d \le y^u$.
This implies,
\begin{equation}
\label{eq:Prelim}
\left|\cP_\pm (\eta,M,q)  - \frac{\varphi(q)}{2q} M
\prod_{\substack{p\le M^\eta\\ p\nmid q}}\(1- \frac{1}{p}\)  \right| \le \Delta_1 + \Delta_2,
\end{equation}
where
$$
 \Delta_1 = C u^{-u/2}  \frac{\varphi(q)}{q} M
\prod_{\substack{p\le M^\eta\\ p\nmid q}}\(1- \frac{1}{p}\)
$$
for some absolute constant $C$,
     and
$$
\Delta_2  \ll \sum_{d \le y^u}|R(q,M,d)|
$$
with $R(q,M,d)$ defined by~\eqref{eq:Ad}.

For $\Delta_1$, recalling the choice of $u$ and $y$, we derive
\begin{equation}
\label{eq:ET1}
 \Delta_1   \le   C \eta^{\eta^{-1/2}/4}  \frac{\varphi(q)}{q} M
\prod_{\substack{p\le M^\eta\\ p\nmid q}}\(1- \frac{1}{p}\) .
\end{equation}

For $\Delta_2$,  using~\eqref{eq:RqMd} and assuming
that $\eta \le \delta/2$,  we obtain
\begin{equation}
\label{eq:ET2}
\Delta_2 \ll \sum_{d \le y^u}  (M/d)^{1-\delta}
\ll  M^{1-\delta/2} \le M^{1-\eta}.
\end{equation}
We also note that
\begin{equation}
\label{eq:MT}
\begin{split}
 \frac{\varphi(q)}{q}
\prod_{\substack{p\le M^\eta\\ p\nmid q}}\(1- \frac{1}{p}\) &
= \prod_{\substack{p\le M^\eta}}\(1- \frac{1}{p}\)
    \prod_{\substack{p> M^\eta\\ p\mid q}}\(1- \frac{1}{p}\) \\
& = \(1+ O(M^{-\eta})\) \prod_{\substack{p\le M^\eta}}\(1- \frac{1}{p}\).
\end{split}
\end{equation}

Thus substituting~\eqref{eq:ET1}, \eqref{eq:ET2} and~\eqref{eq:MT} in~\eqref{eq:Prelim}
and recalling that by the Mertens formula, see~\cite[Section~I.1.6, Theorem~11]{Ten}, we have
$$
\prod_{p\le M^\eta}\(1- \frac{1}{p}\) = \frac{e^{-\gamma} + o(1)}{\eta \log M},
$$
where $\gamma= 0.57721\ldots$ is the Euler constant,
we conclude the proof.
\end{proof}

\begin{cor}
\label{cor:Char P}
 For any $\varepsilon > 0$ there exists some  $\eta_0> 0$ such that
for any positive $\eta < \eta_0$,  integers
$M \ge  q^{1/3+\varepsilon}$,
where $q\ge 2$ is not a perfect square, we have
$$
\left|\sum_{m \in \cP(\eta,M) } \(\frac{m}{q}\) \right|
\le
C_0 \eta^{\eta^{-1/2}/4-1} \frac{M}{\log M} +
O\(M^{1-\eta}\),
$$
where $C_0$  is an absolute constant.
\end{cor}

\section{Proof of Theorem~\ref{thm:dup}}

Let
$$
h = \min_{z \in [Q, 2Q]}  \psi(z).
$$
     We consider the interval  $\cI = [u+1, u+h]$.
Without loss of generality we can assume that, say, $\psi(z) \le \log z$,
so that $h = o(Q)$.

Let us fix some arbitrary $\kappa > 0$, we show that
for all but at most $\kappa Q/\log Q$ primes $p\in [Q, 2Q]$
there is a
quadratic non-residue in $\cI$.

Let $\cN$ be an arbitrary set of integers $n \in \cI$ with either $n \equiv 1 \pmod 4$
or $n \equiv 3 \pmod 4$. So we observe that
\begin{equation}
\label{eq:n1n2}
n_1n_2 \equiv 1 \pmod 4, \qquad n_1,n_2 \in \cN.
\end{equation}

Consider the sum
$$
S = \sum_{p\in [Q, 2Q]} \left| \sum_{n \in \cN}
 \(\frac{n}{p}\)\right|^2
$$
of Legendre symbols.
Clearly, if $\cN$ consists of only quadratic residues (or zeros) modulo $p$
then
$$
 \sum_{n \in \cN}
 \(\frac{n}{p}\)\ge \# \cN - 1.
 $$
Thus
\begin{equation}
\label{eq:S/N}
\#\{p\in [Q, 2Q]~:~d_u(p) \ge h\} \le \frac{S}{(\# \cN-1)^2}.
\end{equation}

We now choose yet another real parameter $\eta> 0$.

 Expanding the summation from primes $p\in [Q, 2Q]$, 
     squaring and extending the summation 
to all integers $m \in \cP(\eta,M)$, we obtain
$$
S\le \sum_{m\in \cP(\eta,M)} \left| \sum_{n \in \cN}
 \(\frac{n}{m}\)\right|^2.
 $$

Squaring and changing the order of summation,
we obtain
$$
S\le \sum_{n_1,n_2 \in \cN}
 \sum_{m\in \cP(\eta,M)}
 \(\frac{n_1n_2}{m}\). $$
Finally, using~\eqref{eq:n1n2}, we derive
 $$
S\le
\sum_{n_1,n_2 \in \cN}
 \sum_{m\in \cP(\eta,M)}
 \(\frac{m}{n_1n_2}\).
$$
If $n_1n_2$ is not a perfect square,
we apply Corollary~\ref{cor:Char P} with
$$q=n_1n_2 \le (u+h)^2 \le 5Q^2
$$
(provided that $Q$ is large enough) to estimate the inner sum.
Otherwise, that is, when  $n_1n_2$ is  a perfect square,
we use the trivial bound $\# \cP(\eta,M)$ for the inner sum,
getting
$$
S \le T \# \cP(\eta,2Q) + h^2 \(C_0 \eta^{\eta^{-1/2}/4-1} \frac{Q}{\log(2Q)} +   O\(Q^{1-\eta}\)\),
$$
where $T$ is the number of products $n_1n_2$ with
$n_1,n_2 \in \cN$ that are perfect  squares.
Thus using~\eqref{eq:Card P}, we see from
we see from~\eqref{eq:S/N}
that
\begin{equation}
\begin{split}
\label{eq:bound}
&\#\{p\in [Q, 2Q]~:~d_u(p) \ge h\} \\
&\qquad  \le c_0  \frac{QT}{\eta  (\# \cN-1)^2 \log Q} \\
& \qquad\qquad +
 \frac{h^2} {(\# \cN-1)^2} \(C_0 \eta^{\eta^{-1/2}/4-1} \frac{Q}{\log(2Q)} +   O\(Q^{1-\eta}\)\),
\end{split}
\end{equation}

We now consider two different choices of the set $\cN$ depending on the relative size
of $u$ and $h$.

If $h \ge u^{1/2}/\log u$, we consider the sets of
$\cN_1$ and $\cN_3$ of square-free integers $n \in \cI$ with  $n \equiv 1 \pmod 4$
and $n \equiv 3 \pmod 4$ respectively. We now define $\cN$ as the
largest set out of $\cN_1$ and $\cN_3$.
We see from Corollary~\ref{cor:SF} that there are
$$
\# \cN_1 + \# \cN_3 \ge (A + o(1)) h.
$$
Hence $\# \cN \ge (A/2 + o(1)) h$. Clearly for two square-free integers $n_1$ and $n_2$
their product is a perfect square only if $n_1 = n_2$.
Hence, $T = \# \cN$ and we see from~\eqref{eq:S/N} and~\eqref {eq:bound}
that in this case
\begin{equation}
\begin{split}
\label{eq:card}
\#\{p\in [Q, 2Q]~:~d_u(p) \ge h\} &\\
\le  C_1 \eta^{-1} \frac{Q}{h \log Q}&  + C_2 \eta^{\eta^{-1/2}/4-1} \frac{Q}{\log Q}
+ C_3  Q^{1-\eta}
\end{split}
\end{equation}
for some absolute constants $C_1$, $C_2$, $C_3$.

We now assume that  $h < u^{1/2}/\log u$.
If  $n_1n_2 = m^2$ for an integer $m$ then,  writing $n_1 = k_1d$,
$n_2 = k_2d$, with $d = \gcd(n_1,n_2)$,
we see that
$$k_1 = m_1^2 \mand k_2  = m_2^2
$$
for some integers $m_1, m_2$.  Assume $m_1 < m_2$.
Thus
$$u/d \le m_1^2 < m_2^2 \le u/d + h/d.
$$
Therefore
$$
(u/d)^{1/2} \ll h/d
$$
or
$$
h \gg (du)^{1/2} \ge u^{1/2},
$$
which contradicts our choice of $h$. So  taking $\cN$ as the set of all integer $n \in \cI$
with $n \equiv 1 \pmod 4$ we see that  $T =\# \cN$ and we
obtain~\eqref{eq:card} again.

We not choose $\eta$ small enough to satisfy
$$
 C_2 \eta^{\eta^{-1/2}/4-1} \le \frac{1}{3}\kappa
$$
then we choose $Q$ large enough to satisfy
$$
 C_1 \eta^{-1}h^{-1}  \le \frac{1}{3}\kappa\mand
 C_3  Q^{1-\eta} \le \frac{1}{3}\kappa.
$$
With these parameters, we derive from~\eqref{eq:card}
that
$$
\#\{p\in [Q, 2Q]~:~d_u(p) \ge h\}  \le \kappa \frac{Q}{\log Q} .
$$
Since $\kappa>0$ is arbitrary, the result now follows.

\section{Comments}

Note that the  inequality $u \le 2Q$ in Theorem~\ref{thm:dup}
is a natural restriction with
respect to primes $p\in [Q, 2Q]$.
On the other hand,  it is also interesting to remove this condition.
It is easy to see that the limit $u \le 2Q$ in Theorem~\ref{thm:dup}
can be increased a little if one uses the full power of the Burgess bound.
In
fact it is easy to see that for quadratic characters only the square-free part
of the modulus $q$ matters so one can actually use  Lemma~\ref{lem:Burg} with
any integer $\nu \ge 1$, see~\cite[Theorem~12.6]{IwKow}.
However for large $u$ one needs some new ideas.

Furthermore,  obtaining a version of Theorem~\ref{thm:dup} with
an unlimited $u$ is essentially equivalent to estimating $d(p)$
for almost all primes $p$. Indeed, assume there are $N$ ``exceptional'' primes
$\ell_1, \ldots, \ell_N\in [Q, 2Q]$ with $d(\ell_i) \ge \psi(\ell_i)$,
$i=1,\ldots, N$, for some function $\psi(z)$.
     This means that
there are integers $u_i$ with
$$
d_{u_i}(\ell_i) \ge \psi(\ell_i), \qquad i=1,\ldots, N.
$$
     Let us
choose an integer $u$ satisfying
$$
u\equiv u_i \pmod{\ell_i},  \qquad i=1,\ldots, N.
$$
     Then we have
$$
d_{u}(\ell_i) = d_{u_i}(\ell_i) \ge \psi(\ell_i), \qquad i=1,\ldots, N.
$$
So a version of Theorem~\ref{thm:dup} with
an unlimited $u$ immediately implies an upper bound on $N$.

Similar questions are also interesting to study for the gaps
between primitive roots modulo $p$.

\section*{Acknowledgements}

The first author was
supported in part   by
the 
Russian Fund for Basic Research, Grant N.~14-01-00332, and by
the 
Program Supporting
Leading Scientific Schools, Grant Nsh-3082.2014.1.

The second author would like to thank the Max Planck Institute for
Mathematics, Bonn, for support and hospitality during his work on this project.
The second author was also
supported in part by 
Australian Research Council, Grants DP110100628 and DP130100237

\end{document}